\documentclass[final,english]{article}

\usepackage{amsfonts,amsmath,amsthm,amscd,amssymb,latexsym,amsbsy,pb-diagram,cite}
\usepackage[dvips]{graphicx}
\usepackage[mathscr]{eucal}

\usepackage[notref,notcite]{showkeys}

\newtheorem{definition}{Definition}[section]
\newtheorem{theorem}[definition]{Theorem}
\newtheorem{lemma}[definition]{Lemma}
\newtheorem{proposition}[definition]{Proposition}

\theoremstyle{definition}
\newtheorem{remark}[definition]{Remark}
\newtheorem{example}[definition]{Example}

\numberwithin{equation}{section}

\DeclareMathOperator{\diam}{diam}

\begin{document}

\title{\bf Metric spaces with unique pretangent spaces}

\author{{\bf Oleksiy Dovgoshey:$(\boxtimes)$} \\\smallskip\\  Institute of Applied Mathematics and
Mechanics of NASU,\\ R.~Luxemburg str. 74, Donetsk 83114, Ukraine
\\ {\it aleksdov@mail.ru}\\\bigskip\\
{\bf Fahreddin Abdullayev} and {\bf Mehmet K\"{u}\c{c}\"{u}kaslan:}\\\smallskip\\
Mersin University Faculty of Literature and Science,\\ Department
of Mathematics,\\ 33342 Mersin, Turkey\\{\it fabdul@mersin.edu.tr}
and {\it mkucukaslan@mersin.edu.tr}}

\date{}

\maketitle
\begin{abstract}
We find necessary and sufficient conditions under which an
arbitrary metric space $X$ has a unique pretangent space at the
marked point $a\in X$.
{\bf 2000 Mathematic Subject Classification:} 54E35
\end{abstract}

\section{Introduction}

Analysis on metric spaces with no a priory smooth structure has
rapidly developed the present time. This development is closely
related to some generalizations of the differentiability.
Important examples of such generalizations and even an axiomatics
of so-called ``pseudo-gradients'' can be found in
\cite{Am,AmKi1,AmKi2,Ch,Ha,He,HeKo,Sh} and respectively in
\cite{Am1}. In almost all above-mentioned books and papers the
generalized differentiations involve an induced linear structure
that makes possible to use the classical differentiations in the
linear normed spaces. A new {\it intrinsic} approach to the
introduction of the ``smooth'' structure for general metric spaces
was proposed  by O.~Martio and by the first author of the present
paper in \cite{DM}.

A basic technical tool in \cite{DM} is the notion of  pretangent
spaces at a point $a$ of an arbitrary metric space $X$ which were
defined as  factor spaces of  families of sequences of points
$x_n\in X$  convergent to $a$. In  present paper we find and prove
necessary and sufficient conditions under which the metric space
with a marked point $a$ has a unique pretangent space at $a$ for
every normalizing sequence $\tilde r$, see Definition \ref{1:d1.3}
below.

For convenience we recall the main notions form \cite{DM}, see
also \cite{Dov}.

Let $(X,d)$ be a metric space and let $a$ be point of $X$. Fix a
sequence $\tilde r$ of positive real numbers $r_n$ which tend to
zero. In what follows this sequence $\tilde r$ be called a
\textit{normalizing sequence}. Let us denote by  ${\tilde X}$ the
set of all sequences of points from $X$.

\begin{definition}
\label{1:d1.1} Two sequences $\tilde{x},\tilde{y}\in \tilde X$,
$\tilde x=\{ x_n\}_{n\in \mathbb{N}}$ and $\tilde y=\{ y_n\}_{n\in
\mathbb{N}}$,  are mutually stable (with respect to a normalizing
sequence $\tilde r=\{ r_n\}_{n\in \mathbb{N}}$) if there is a
finite limit
\begin{equation}  \label{1:eq1.1}
\lim_{n\to\infty} \frac{d(x_n, y_n)}{r_n} := \tilde d_{\tilde
r}(\tilde x, \tilde y)=\tilde d(\tilde x, \tilde y).
\end{equation}
\end{definition}

We shall say that a family $\tilde F\subseteq \tilde X$ is
\textit{self-stable} (w.r.t.  $\tilde r$) if every two $\tilde
x,\tilde y \in \tilde F$ are mutually stable. A family $\tilde F
\subseteq\tilde X$ is \emph{maximal self-stable} if $\tilde{F}$ is
self-stable and for an arbitrary $\tilde z\in \tilde X$ either
$\tilde z\in \tilde F$ or there is $\tilde x\in \tilde F$ such
that $\tilde x$ and $\tilde z$ are not mutually stable.

A standard application of Zorn's Lemma leads to the following

\begin{proposition}
\label{1:p1.2} Let $(X, d)$ be a metric space and let $a\in X$.
Then for every normalizing sequence $\tilde r=\{ r_n\}_{n\in
\mathbb{N}}$ there exists a maximal self-stable family $\tilde
X_a=\tilde X_{a, \tilde r}$ such that $\tilde a:=\{ a, a,
\dots\}\in \tilde X_a$.
\end{proposition}

Note that the condition $\tilde a \in \tilde X_a$ implies the
equality
\begin{equation*}
\lim_{n\to\infty} d(x_n, a)=0
\end{equation*}
for every $\tilde x=\{ x_n\}_{n\in \mathbb{N}}$ which belongs to
$\tilde X_a$.

Consider a function $\tilde d : \tilde X_a\times \tilde X_a \to
\mathbb{R}$ where $\tilde d(\tilde x, \tilde y)=\tilde d_{\tilde
r}(\tilde x, \tilde y)$ is defined by \eqref{1:eq1.1}. Obviously,
$\tilde d$ is symmetric and nonnegative. Moreover, the triangle
inequality for the original metric $d$ implies
\begin{equation*}
\tilde d (\tilde x, \tilde y) \leq \tilde d(\tilde x, \tilde
z)+\tilde d(\tilde z, \tilde y)
\end{equation*}
for all $\tilde x, \tilde y, \tilde z$ from $\tilde X_a$. Hence
$(\tilde X_a, \tilde d)$ is a pseudometric space.

\begin{definition}
\label{1:d1.3} The pretangent space to the space $X$ at the point
$a$ w.r.t. a normalizing sequence $\tilde r$ is the metric
identification of the pseudometric space $(\tilde X_{a, \tilde r},
\tilde d)$.
\end{definition}

Since the notion of pretangent space is basic for the present
paper, we remaind this metric identification construction.

Define a relation $\sim$ on $\tilde X_a$ by $\tilde x \sim \tilde
y$ if and only if $\tilde d(\tilde x, \tilde y)=0$. Then $\sim $
is an equivalence relation. Let us denote by
$\Omega_{a}=\Omega_{a, \tilde r}=\Omega_{a,\tilde r}^X$ the set of
equivalence classes in $\tilde X_a$ under the equivalence relation
$\sim$. It follows from general properties of pseudometric spaces,
see, for example, \cite[Chapter 4, Th.~15]{Kell}, that if $\rho$
is defined on $\Omega_a$ by
\begin{equation}  \label{1:eq1.2}
\rho(\alpha, \beta) :=\tilde d(\tilde x, \tilde y)
\end{equation}
for $\tilde x\in \alpha$ and $\tilde y \in \beta$, then $\rho$ is
the well-defined metric on $\Omega_a$. The metric identification
of $(\tilde X_a, \tilde d)$ is, by definition, the metric space
$(\Omega_a, \rho)$.

Remark that $\Omega_{a, \tilde r} \neq\emptyset$  because the
constant sequence $\tilde a$ belongs to $\tilde X_{a,\tilde r}$,
see Proposition \ref{1:p1.2}.

Let $\{n_k\}_{k\in \mathbb{N}}$ be an infinite, strictly
increasing sequence of natural numbers. Let us denote by $\tilde
r^{\prime}$ the subsequence $\{ r_{n_k}\}_{k\in\mathbb{N}}$ of the
normalizing sequence $\tilde r=\{ r_n\}_{n\in \mathbb{N}}$ and let
$\tilde x^{\prime}:=\{ x_{n_k}\}_{k\in \mathbb{N}}$ for every
$\tilde x=\{x_n\}_{n\in \mathbb{N}}\in\tilde{X}$.  It is clear
that if $\tilde x$ and $\tilde y $ are mutually stable w.r.t.
$\tilde r$, then $\tilde x^{\prime}$ and $\tilde y^{\prime}$ are
mutually stable w.r.t. $\tilde r^{\prime}$ and
\begin{equation}  \label{1:eq1.3}
\tilde d_{\tilde r} (\tilde x, \tilde y)=\tilde d_{\tilde
r^{\prime}}( \tilde x^{\prime},\tilde y^{\prime}).
\end{equation}
If $\tilde X_{a, \tilde r}$ is a maximal self-stable (w.r.t.
$\tilde r$) family, then, by Zorn's Lemma, there exists a maximal
self-stable (w.r.t. $\tilde r^{\prime}$) family $\tilde X_{a,
\tilde r^{\prime}}$ such that
\begin{equation*}
\{ \tilde x^{\prime}:\tilde x \in \tilde X_{a, \tilde
r}\}\subseteq \tilde X_{a, \tilde r^{\prime}}.
\end{equation*}
Denote by $\mathrm{in}_{\tilde r^{\prime}}$ the mapping from
$\tilde X_{a, \tilde r}$ to $\tilde X_{a, \tilde r^{\prime}}$ with
$\mathrm{in}_{\tilde r^{\prime}}(\tilde x)=\tilde x^{\prime}$ for
all $\tilde x\in \tilde X_{a, \tilde r}$. If follows from
\eqref{1:eq1.2} that after the metric identifications
$\mathrm{in}_{\tilde r^{\prime}}$ pass to an isometric embedding
$\mathrm{em}^{\prime}$: $\Omega_{a, \tilde r}\to \Omega_{a, \tilde
r^{\prime}}$ under which the diagram
\begin{equation}  \label{1:eq1.4}
\begin{array}{ccc}
\tilde X_{a, \tilde r} & \xrightarrow{\ \ \mbox{in}_{\tilde r'}\ \
} &
\tilde X_{a, \tilde r^{\prime}} \\
\!\! \!\! \!\! \!\! \! p\Bigg\downarrow &  & \! \!\Bigg\downarrow
p^{\prime}
\\
\Omega_{a, \tilde r} & \xrightarrow{\ \ \mbox{em}'\ \ \ } &
\Omega_{a, \tilde r^{\prime}}
\end{array}
\end{equation}
is commutative. Here $p$, $p^{\prime}$ are metric identification
mappings, $p(\tilde x):=\{\tilde y\in \tilde X_{a, \tilde
r}:\tilde d_{\tilde r}(\tilde x,\tilde y)=0\}$ and
$p^{\prime}(\tilde x):=\{ \tilde y\in\tilde X_{a,\tilde
r^{\prime}}:\tilde d_{\tilde r^{\prime}}(\tilde x,\tilde y)=0\}$.

Let $X$ and $Y$ be two metric spaces. Recall that a map $f:X\to Y$
is called an \textit{isometry} if $f$ is distance-preserving and
onto.

\begin{definition}
\label{1:d1.4} A pretangent $\Omega_{a, \tilde r}$ is tangent if
$\mathnormal{em}^{\prime}$: $\Omega_{a, \tilde r}\to\Omega_{a,
\tilde r^{\prime}}$ is an isometry for every $\tilde r^{\prime}$.
\end{definition}

Simple arguments give the following proposition.

\begin{proposition}
\label{1:p1.5} Let $X$ be a metric space with a marked point $a$,
$ \tilde{r}$ a normalizing sequence and $\tilde{X}_{a,\tilde{r}}$
a maximal self-stable family with correspondent pretangent space
$\Omega _{a,\tilde{r}}$. The following statements are equivalent.

(i) $\Omega_{a,\tilde{r}}$ is tangent.

(ii) For every subsequence $\tilde{r}^{\prime }$ of the sequence $
\tilde{r}$ the family $\{ \tilde{x}^{\prime }:\tilde{x}\in
\tilde{X}_{a,\tilde{r}}\}$ is maximal self-stable w.r.t.
$\tilde{r}^{\prime}$.

(iii) A function $em^{\prime }:\Omega_{a,\tilde{r}}\longrightarrow
\Omega _{a,\tilde{r}^{\prime }}$ is surjective for every
$\tilde{r} ^{\prime}$.

(iv) A function $in_r':\,\tilde{X}_{a,\tilde r} \longrightarrow
\tilde{X}_{a,\tilde{r}^{\prime}}$ is surjective for every
$\tilde{r}'$.
\end{proposition}

For the proof see \cite[Proposition 1.2]{Dov} or \cite[Proposition
1.5]{DAK1}.

\section{Conditions of uniqueness of pretangent spaces}

In this section we start from the simplest example of a metric
space with unique pretangent spaces.
\begin{example}
\label{4:e2.1} Let $X=\mathbb{R}^{+}=\left[ 0,\infty \right[$ be
the set of all non-negative, real numbers with the usual metric
\begin{equation*}
d(x,y)=\left\vert x-y\right\vert,
\end{equation*}
let $\tilde{r}=\left\{ r_{n}\right\}_{n\in\mathbb{N}}$  be an
arbitrary normalizing sequence and let $0$ be the marked point of
$X$. Consider a maximal self-stable family
$\tilde{X}_{0,\tilde{r}}$.
\end{example}

\begin{proposition}
\label{4:p2.2} The following statements are true.

(i) Let $\tilde{x}=\left\{x_{n}\right\}_{n\in\mathbb{N}}\in
\tilde{X}$. Then $\tilde{x}\in \tilde{X}_{0,\tilde{r}}$ if and
only if there is $c\geq 0$ such that
\begin{equation}
\underset{n\rightarrow \infty}{\lim }\frac{x_{n}}{r_{n}}=c.
\label{4:eq1.5}
\end{equation}

(ii) For every two
$\tilde{x}=\left\{x_{n}\right\}_{n\in\mathbb{N}},\tilde{y}=\left\{
y_{n}\right\}_{n\in\mathbb{N}}$ from $\tilde{X}_{0,\tilde{r}}$ the
equality
\begin{equation*}
\tilde{d}_{\tilde{r}}(\tilde{x},\tilde{y})=0
\end{equation*}
holds if and only if
\begin{equation*}
\underset{n\rightarrow \infty }{\lim
}\frac{x_{n}}{r_{n}}=\underset{ n\rightarrow\infty}{\lim
}\frac{y_{n}}{r_{n}}.
\end{equation*}

(iii) The pretangent space $\Omega_{0,\tilde{r}}$ corresponding to
$ \tilde{X}_{0,\tilde{r}}$ is isometric to
$(\mathbb{R}^{+},\left\vert .,.\right\vert)$.

(iv) The pretangent space $\Omega_{0,\tilde{r}}$ is tangent.
\end{proposition}

\begin{proof}
(i) If $\tilde{x}=\left\{ x_{n}\right\} _{n\in\mathbb{N}}\in
\tilde{X}_{0,\tilde r}$, then there is a finite limit
\begin{equation*}
\underset{n\rightarrow \infty }{\lim }\frac{\left\vert
x_{n}-0\right\vert}{r_{n}}=\tilde d(\tilde{x},\tilde{0}).
\end{equation*}%
Since we have $x_{n}=$ $\left\vert x_{n}-0\right\vert $ for all
$n\in\mathbb{N}$, the limit relation \eqref{4:eq1.5} holds with
$c=\tilde d(\tilde{x},\tilde{0})$. Suppose that $\tilde x,\tilde
y\in\tilde X$, $\tilde{x}=\left\{ x_{n}\right\}
_{n\in\mathbb{N}},\ \tilde{y}=\left\{ y_{n}\right\}
_{n\in\mathbb{N}}$ and there are $c_{1},c_{2}\in\mathbb{R}^{+}$
such that
\begin{equation*}
\underset{n\rightarrow \infty }{\lim
}\frac{x_{n}}{r_{n}}=c_{1},~\underset{n\rightarrow \infty}{\lim
}\frac{y_{n}}{r_{n}}=c_{2}.
\end{equation*}
It implies that
\begin{equation}
\underset{n\rightarrow \infty }{\lim }\frac{\left\vert
x_{n}-y_{n}\right\vert}{r_{n}}=\left\vert c_{1}-c_{2}\right\vert,
\label{4:eq1.6}
\end{equation}
so $\tilde{x}~\text{and}~\tilde{y}$ are mutually stable. It
implies Statement (i).

(ii) Statement (ii) follows from Statement (i) and
\eqref{4:eq1.6}.

(iii) Define a function $f:\Omega
_{0,\tilde{r}}\rightarrow\mathbb{R}^{+}$ by the rule: If $\beta
\in \Omega_{0,\tilde{r}}$ and $\tilde{x}\in \beta$, then write
$f(\beta ):=\underset{n\rightarrow \infty }{\lim}
\frac{x_{n}}{r_{n}}$. Statements (i),(ii) and limit relation
\eqref{4:eq1.6} imply that $f$ is a well-defined isometry.

(iv) Let $\tilde{n}=\left\{n_{k}\right\}_{k\in\mathbb{N}}$ be a
strictly increasing, infinite sequence of natural numbers and let
$ \tilde{r}^{\prime }=\left\{ r_{n_k}\right\}_{k\in\mathbb{N}}$ be
the corresponding subsequence of the normalizing sequence
$\tilde{r}$. If  $\tilde{x} =\left\{ x_{k}\right\}
_{k\in\mathbb{N}}\in \tilde{X}_{0,\tilde{r}'}$ then, by Statement
(i), there is $b\in\mathbb{R}^{+}$ such that
\begin{equation*}
\underset{k\rightarrow \infty }{\lim }\frac{x_{k}}{r_{n_{k}}}=b.
\end{equation*}%
Define  $~\tilde{y}=\left\{y_{n}\right\}_{n\in\mathbb{N}}\in
\tilde{X}$ \ by the rule
\begin{equation*}
y_{n}:=
\begin{cases}
x_{k}& \text{if there is an element }n_{k}\text{ of the sequence }\tilde{n}\text{ such that }n_{k}=n, \\
br_{n}& \text{otherwise}.
\end{cases}
\end{equation*}
It is  clear that $\tilde y^{\prime}=\left\{ y_{n_{k}}\right\}
_{k\in\mathbb{N}}=\tilde{x}$ and
\begin{equation*}
\underset{n\rightarrow \infty }{\lim }\frac{y_{n}}{r_{n}}=b.
\end{equation*}%
Hence, by Statement (i),  $\tilde{y}$ belongs to
$\tilde{X}_{0,\tilde{r}}$. Using Proposition \ref{1:p1.5} we see
that $\Omega _{0,\tilde{r}}$ is tangent.
\end{proof}

Statement (i) of Proposition \ref{4:p2.2} shows that the space
$(\mathbb{R}^{+},\left\vert .,.\right\vert)$ possesses an
interesting property: For every normalizing sequence $\tilde{r}$
there exists a unique pretangent space $\Omega_{0,\tilde r}$. The
main theorem of this paper describes metric spaces which have this
property.
\begin{remark}\label{r:2.3}
The uniqueness in the previous paragraph and in  Theorem
\ref{5:t2.5} below is understood in the usual set-theoretical
sense. Statement (i) of Proposition \ref{4:p2.2} implies that for
$X=\mathbb R^+$ the family (= the set) $\tilde X_{a,\tilde r}$ is
unique. Hence $\Omega_{0,\tilde r}$, the metric identification of
$\tilde X_{0,\tilde r}$, is also unique. Since
$$
\tilde X_{0,\tilde r}=\cup\{\tilde
x\in\beta:\beta\in\Omega_{0,\tilde r}\},
$$
i.e., the set $\tilde X_{0,\tilde r}$ is the union of all
equivalence classes $\beta\in\Omega_{0,\tilde r}$, the uniqueness
of the pretangent spaces $\Omega_{0,\tilde r}$ gives the
uniqueness of $\tilde X_{0,\tilde r}$.
\end{remark}
Let $(X,d)$ be a metric space with marked point $a$. For each pair
of nonvoid sets $C,D\subseteq X$ write
\begin{equation*}
\Delta (C,D):=\sup \left\{ d(x,y):x\in C,y\in D\right\}, \quad
\delta (C,D):=\inf \left\{ d(x,y):x\in C,y\in D\right\}
\end{equation*}%
and write
\begin{equation*}
A_{a}(r,k):=\left\{ x\in X:\frac{r}{k}\leq d(x,a)\leq rk\right\} ,
\quad S_{a}(r):=\left\{ x\in X:d(x,a)= r\right\}
\end{equation*}%
and for every $r>0$ and every $k\geq 1$ define
\begin{equation*}
R_{a,X}:=\left\{ r\in\mathbb{R}^{+}:S_{a}(r)\neq \emptyset
\right\}
\end{equation*}%
and for every $\varepsilon \in \left] 0,1\right[ $%
\begin{equation*}
R_{\varepsilon }^{2}:=\left\{ (r,t)\in R^{2}_{a,X}:r\neq 0\neq
t\text{ and }\left\vert \frac{r}{t}-1\right\vert \geq \varepsilon
\right\}
\end{equation*}%
where $R^{2}_{a,X}$ is the Cartesian product of $R_{a,X}$'s. See
Fig.~1.

\suppressfloats[t]

$\phantom{a}$
\begin{figure}[h!]\centering
\includegraphics[width=12cm,keepaspectratio]{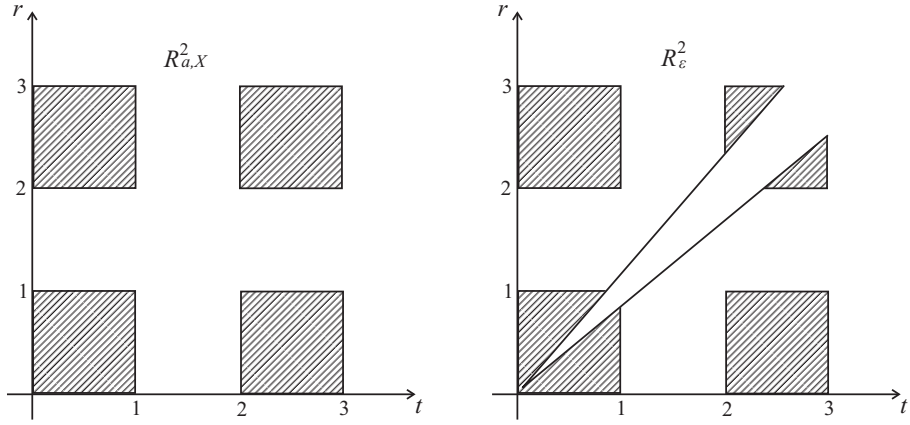}
\caption{The sets $R^2_{a,X}$ and $R^2_\varepsilon$ with
$R_{a,X}=[0,1]\cup[2,3]$ and $\varepsilon=\frac16$. Nontengential
limit \eqref{5:eq2.9} is taken over the set $R^2_\varepsilon$.}
\end{figure}

\begin{theorem}
\label{5:t2.5} Let $(X,d)$ be a metric space and let $a$ be a
limit point of $X$. Then for every normalizing sequence $\tilde r$
there is a unique pretangent space $\Omega_{a,\tilde r}$ if and
only if the following three conditions are satisfied
simultaneously.

(i) The limit relation
\begin{equation}\label{5:eq2.8}
\underset{k\rightarrow 1}{\lim }\underset{r\rightarrow 0}{\lim
\sup }\frac{\diam(A_{a}(r,k))}{r}=0\qquad r\in\,]0,\infty[\,,\
k\in[1,\infty[
\end{equation}%
holds.

(ii)  We have
\begin{equation}\label{5:eq2.9}
\underset{(t,g)\in R_{\varepsilon
}^{2}}{\underset{(t,g)\rightarrow (0,0)}{\lim }}\frac{\Delta
(S_{a}(g),S_{a}(t))}{\delta (S_{a}(g),S_{a}(t))}=1
\end{equation}
for every $\varepsilon \in \left] 0,1\right[ \,$.

(iii) If $\left\{(q_{n},t_{n})\right\}_{n\in\mathbb{N}} $ is a
sequence such that $ (q_{n},t_{n})\in R_{\varepsilon }^{2}$ for
all $n\in\mathbb{N}$ and
\[
\underset{n\rightarrow \infty }{\lim }(q_{n},t_{n})=(0,0)
\]%
and there is
\begin{equation}
\underset{n\rightarrow \infty }{\lim }\frac{q_{n}}{t_{n}}=c_{0}\in
\left[ 0,\infty \right], \label{5:eq2.92}
\end{equation}
then there exists a finite limit
\begin{equation}
\underset{n\rightarrow \infty }{\lim }\frac{\Delta
(S_{a}(q_{n}),S_{a}(t_{n}))}{\left\vert q_{n}-t_{n}\right\vert
}:=\varkappa _{0}.  \label{5:eq2.91}
\end{equation}%
\end{theorem}

\begin{remark}\label{r:2.5}
 The annulus $A_a(r,k)$ can be void in \ref{5:eq2.8}. At that time
we use the convention
$$
\diam A_r(r,k)=\diam(\emptyset)=0.
$$
\end{remark}

We need the following lemma.

\begin{lemma}
\label{5:l2.6} Let $(X,d)$ be a metric space with a marked point
$a$. A pretangent space $\Omega_{a,\tilde r}$ is unique for every
normalizing sequence $\tilde r$ if and only if the implication
\begin{equation*}
\text((\tilde{x}\text{ and }\tilde{a}\text{ are  mutually
stable})\ \&\ (\tilde{y}\text{ and }\tilde{a}\text{ are mutually
stable}))
\end{equation*}
\begin{equation}
\Longrightarrow (\tilde{x}\text{ and }\tilde{y}\text{ are mutually
stable})  \label{5:eq2.10}
\end{equation}%
is true for every $\tilde{x}$, $\tilde{y}$ $\in \tilde{X}.$
\end{lemma}

\begin{proof}
Suppose that \eqref{5:eq2.10} is true. Let $\tilde{X}_{a,\tilde{r}%
}^{m}$ be the set of all $\tilde{x}\in \tilde{X}$ \ which are
mutually stable with $\tilde{a}$. It follows from \eqref{5:eq2.10}
that $\tilde X^{m}_{a,\tilde r}$ is self-stable. Consider an
arbitrary maximal self-stable $\tilde{X}_{a,\tilde{r}}$, then, by
 definition of $\tilde{X}_{a,\tilde{r}}$, we obtain the
inclusion $\tilde{X} _{a,\tilde{r}}^{m}\supseteq
\tilde{X}_{a,\tilde{r}}$. Since $ \tilde{X}_{a,\tilde{r}}$ is
maximal self-stable, we have also $
\tilde{X}_{a,\tilde{r}}\supseteq \tilde{X}_{a,\tilde{r}}^{m}$.
Hence the equality
\begin{equation*}
\tilde{X}_{a,\tilde{r}}=\tilde{X}_{a,\tilde{r}}^{m}
\end{equation*}%
holds for all $\tilde{X}_{a,\tilde{r}}$, so all $\tilde{X}_{a,
\tilde{r}}$ coincide.

Now suppose that $\tilde X_{a,\tilde r}$ is unique for every
$\tilde r$ and there are $\tilde{x},\tilde{y}\in \tilde{X}$ and
there is a normalizing sequence $\tilde t$ such that:

$\tilde{x}$ and $\tilde{a}$ are \ mutually stable;

$\tilde{y}$ and $\tilde{a}$ are \ mutually stable;

$\tilde{x}$ and $\tilde{y}$ are not mutually stable. By Zorn's
Lemma there exist maximal self-stable families
$\tilde{X}_{a,\tilde{t}}^{(1)}\supseteq
\left\{\tilde{x},\tilde{a}\right\}$ and
$\tilde{X}_{a,\tilde{t}}^{(2)}\supseteq \left\{ \tilde{y},
\tilde{a}\right\}$. It is clear that $\tilde{X}_{a,\tilde{t}
}^{(1)}\neq $ $\tilde{X}_{a,\tilde{t}}^{(2)}.$ Hence, the
uniqueness of pretangent spaces, see Remark \ref{r:2.5}, implies
\eqref{5:eq2.10}.
\end{proof}

\begin{proof}[\textit{Proof of Theorem \ref{5:t2.5}}]
Assume that $\Omega_{a,\tilde r}$ is unique. We need to verify the
conditions (i)--(iii).

(i) Consider a function $f:[1,\infty[\,\to\mathbb R^+$,
\begin{equation*}
f(k):=k\,\underset{r\rightarrow 0}{\lim \sup
}\frac{\diam(A_{a}(r,k))}{r}.
\end{equation*}%
Since
\begin{equation*}
f(k):=\underset{r\rightarrow 0}{\lim \sup
}\frac{\diam(A_{a}(k\frac rk,k))}{\frac
rk}=\limsup_{t\to0}\frac{\diam(A_a(kt,k))}t
\end{equation*}%
and
$$
A_a(kt,k)=\{x\in X:t\leq d(x,a)\leq k^2t\},
$$
the function $f$ is increasing. Since we have
\begin{equation*}
\frac{\diam(A_{a}(r,k))}{r}\leq \frac{2rk}{r}=2k
\end{equation*}%
for every $k\geq 1$ and all $r>0$, the double inequality
$$
0\leq f(k)\leq2k^2
$$
holds. Consequently
 there is a finite, positive limit $\underset{%
k\rightarrow 1}{\lim }f(k):=c_0$. It is clear that this limit
coincides with the limit in \eqref{5:eq2.8}. Suppose that $c_0>0$.
Let $\varepsilon \in \left] 0,c_{0}\right[ \,$. Then there is
$k_0>1$ such that the double inequality
\begin{equation}
c_{0}-\varepsilon <\limsup_{r\to0}
\frac{\text{diam}(A_{a}(r,k))}{r_{n}}<c_0+\varepsilon
\label{5:eq2.11}
\end{equation}%
holds for all  $k\in \left] 1,k_{0}\right]$. Let $\left\{
k_{n}\right\} _{n\in\mathbb{N}}$ be a strictly decreasing sequence
of real numbers such that all $k_{n}\in\left] 1,k_{0}\right] $ and
\begin{equation}
\underset{n\rightarrow \infty }{\lim }k_{n}=1.  \label{5:eq2.12}
\end{equation}%
Double inequality \eqref{5:eq2.11} implies that there is a
sequence $\tilde r=\{r_n\}_{n\in\mathbb N},\ r_n=r_n(k_n)>0$, such
that $\lim_{n\to0}r_n=0$ and
\begin{equation}
c_{0}-\varepsilon < \frac{\text{diam}(A_{a}(r_{n},k_{n}))}{r_{n}}<
c_0+\varepsilon \label{5:eq2.13}
\end{equation}%
for all $n\in\mathbb{N}$. It follows from \eqref{5:eq2.13} that
there are $\tilde{x}=\left\{ x_{n}\right\} _{n\in\mathbb{N}}$ and
$\tilde{y}=\left\{ y_{n}\right\} _{n\in\mathbb{N}}$ from
$\tilde{X}$ such that
\begin{equation}
x_{n},y_{n}\in A_{a}(r_{n},k_{n})\text{ and
}\frac{d(x_{n},y_{n})}{r_{n}} \geq c_{0}-\varepsilon
\label{5:eq2.14}
\end{equation}%
for all $n\in\mathbb{N}$. The definition of the annulus
$A_{a}(r_{n},k_{n})$ and \eqref{5:eq2.14} imply that
\begin{equation}
\frac{d(x_{n},a)}{r_{n}},\frac{d(y_{n},a)}{r_{n}}\in \left[ \frac{1}{k_{n}}%
,k_{n}\right]  \label{5:eq2.15}
\end{equation}%
for all $n\in\mathbb{N}$. Define a sequence $\tilde{z}=\left\{
z_{n}\right\} _{n\in\mathbb{N}}\in \tilde{X}$ by the rule
\begin{equation}
z_{n}:=\left\{
\begin{array}{ccc}
x_{n} & \text{if} & n\text{ is even} \\
y_{n} & \text{if} & n\text{ is odd. }
\end{array}%
\right.  \label{5:eq2.16}
\end{equation}%
Then it follows from \eqref{5:eq2.12},  \eqref{5:eq2.15} and
\eqref{5:eq2.16} that
\begin{equation*}
\underset{n\rightarrow \infty }{\lim }\frac{d(x_{n},a)}{r_{n}}=\underset{%
n\rightarrow \infty }{\lim }\frac{d(z_{n},a)}{r_{n}}=1.
\end{equation*}%
Moreover \eqref{5:eq2.13} and \eqref{5:eq2.15} imply that
\begin{equation*}
\underset{n\rightarrow \infty }{\lim \inf }\frac{d(x_{n},z_{n})}{r_{n}}=%
\underset{n\rightarrow \infty }{\lim
}\frac{d(x_{2n},z_{2n})}{r_{2n}}=0
\end{equation*}%
but%
\begin{equation*}
\underset{n\rightarrow \infty }{\lim \sup }\frac{d(x_{n},z_{n})}{r_{n}}=%
\underset{n\rightarrow \infty }{\lim \sup }\frac{d(x_{2n+1},z_{2n+1})}{%
r_{2n+1}}\geq c_{0}-\varepsilon >0.
\end{equation*}%
Thus $\tilde{x}$ and $\tilde{a}$ are mutually stable, $\tilde{z}
$ and $\tilde{a}$ are mutually stable but $\tilde{x}$ and $%
\tilde{z}$ are not mutually stable (w.r.t. the normalizing
sequence $\tilde{r}=\left\{ r_{n}\right\} _{n\in\mathbb{N}}$).
Hence, by Lemma \ref{5:l2.6}, pretangent spaces to $X$ at the
point $a$ are  not unique contrary to the assumption.

(ii) Let $\varepsilon\,]0,1[\,$. Since
\begin{equation*}
\Delta (C,D)\geq \delta(C,D)
\end{equation*}%
for all nonvoid sets $C,D\subseteq X,$ we have
\begin{equation*}
1\leq \underset{(t,r)\rightarrow (0,0)}{\lim \inf }\frac{\Delta
(S_{a}(r),S_{a}(t))}{\delta (S_{a}(r),S_{a}(r))}\leq \underset{%
(t,r)\rightarrow (0,0)}{\lim \sup }\frac{\Delta
(S_{a}(r),S_{a}(t))}{\delta (S_{a}(r),S_{a}(t))}:=s_{0}
\end{equation*}
where the upper and lower limits are taken over the set
$R^{2}_\varepsilon$. Hence, it is sufficient to show that
$s_{0}=1$ in the last limit relation. Let $\tilde{r}=\left\{
r_{n}\right\} _{n\in\mathbb{N}}$ and $\tilde{t}=\left\{
t_{n}\right\} _{n\in\mathbb{N}}$ be two sequences of positive real
numbers such that $(r_{n},t_{n})\in R_{\varepsilon }^{2}$ for all
$n\in\mathbb{N}$, and $\underset{n\rightarrow \infty }{\lim
}r_{n}=\underset{n\rightarrow \infty }{\lim }t_{n}=0$ and
\begin{equation}
\frac{\Delta (S_{a}(r_{n}),S_{a}(t_{n}))}{\delta (S_{a}(r_{n}),S_{a}(t_{n}))}%
\rightarrow s_{0}  \label{5:eq2.17}
\end{equation}%
when $n\rightarrow \infty$. Without loss of generality we may suppose that%
\begin{equation*}
0<r_{n}<t_{n}<1
\end{equation*}%
for all $n\in\mathbb{N}$ and there is the limit
\begin{equation}
\underset{n\rightarrow \infty }{\lim }\frac{r_{n}}{t_{n}}:=\gamma
_{0}.  \label{5:eq2.19}
\end{equation}%
First consider the case where $\gamma_{0}=0$. The triangle
inequality implies that
\begin{equation}
r_{n}+t_{n}\geq \Delta (S_{a}(r_{n}),S_{a}(t_{n}))\geq \delta
(S_{a}(r_{n}),S_{a}(t_{n}))\geq t_{n}-r_{n}>0. \label{5:eq2.18}
\end{equation}%
Hence,%
\begin{equation}
\frac{r_{n}+t_{n}}{t_{n}-r_{n}}\geq \frac{\Delta (S_{a}(r_{n}),S_{a}(t_{n}))%
}{\delta (S_{a}(r_{n}),S_{a}(t_{n}))}\geq 1,  \label{5:eq2.20}
\end{equation}%
from this, \eqref{5:eq2.17} and \eqref{5:eq2.18} we obtain%
\begin{equation}
\frac{1+\gamma _{0}}{1-\gamma _{0}}\geq s_{0}\geq 1.
\label{5:eq2.21}
\end{equation}%
Since $\gamma _{0}=0,$ we see that%
\begin{equation*}
s_{0}=1.
\end{equation*}%
Assume now that $0<\gamma _{0}<1.$ (Note that equality $\gamma
_{0}=1$ contradicts the definition of the set $R_{\varepsilon
}^{2}$.) There exist sequences $\tilde{x}=\left\{ x_{n}\right\}
_{n\in\mathbb{N}},\tilde{y}=\left\{ y_{n}\right\}
_{n\in\mathbb{N}}$, $\tilde{z}=\left\{ z_{n}\right\} _{n\in
\mathbb{N} }$ and $\tilde{w}=\left\{ w_{n}\right\} _{n\in
\mathbb{N}
}$ from $\tilde{X}$ which satisfy the following conditions:%
\begin{equation*}
x_{n}\text{ and }y_{n}\text{ belong to }S_{a}(r_{n})\text{ for all
}n\in \mathbb{N} ;
\end{equation*}%
\begin{equation*}
z_{n}\text{ and }w_{n}\text{ belong to }S_{a}(t_{n})\text{ for all
}n\in \mathbb{N} ;
\end{equation*}%
\begin{equation}
\underset{n\rightarrow \infty }{\lim }\frac{\Delta
(S_{a}(r_{n}),S_{a}(t_{n}))}{d(x_{n},z_{n})}=1;  \label{5:eq2.22}
\end{equation}%
\begin{equation}
\underset{n\rightarrow \infty }{\lim }\frac{\delta
(S_{a}(r_{n}),S_{a}(t_{n}))}{d(y_{n},w_{n})}=1.  \label{5:eq2.23}
\end{equation}%
Define new sequences $\tilde{x}^{\ast }=\left\{ x_{n}^{\ast
}\right\} _{n\in \mathbb{N} }$ and $\tilde{z}^{\ast }=\left\{
z_{n}^{\ast }\right\} _{n\in \mathbb{N}
}$ by the rules:%
\begin{equation*}
z_{n}^{\ast }:=\left\{
\begin{array}{ccc}
z_{n} & \text{if} & n\text{ is even,} \\
w_{n} & \text{if} & n\text{ is odd}%
\end{array}%
\right. \,\text{and}\,\text{ \ }x_{n}^{\ast }:=\left\{
\begin{array}{ccc}
x_{n} & \text{if} & n\text{ is even,} \\
y_{n} & \text{if} & n\text{ is odd.}%
\end{array}%
\right.
\end{equation*}%

Relation \eqref{5:eq2.19} and definitions of $\tilde{x},\tilde{y},%
\tilde{z},\tilde{w}$, $\tilde{z}^{\ast }$, $\tilde{x}%
^{\ast}$ imply that%
\begin{equation*}
\underset{n\rightarrow \infty }{\lim }\frac{d(z_{n},a)}{t_{n}}=\underset{%
n\rightarrow \infty }{\lim
}\frac{d(w_{n},a)}{t_{n}}=\underset{n\rightarrow \infty }{\lim
}\frac{d(z_{n}^{\ast },a)}{t_{n}}=1
\end{equation*}%
and%
\begin{equation*}
\underset{n\rightarrow \infty }{\lim }\frac{d(x_{n},a)}{t_{n}}=\underset{%
n\rightarrow \infty }{\lim
}\frac{d(y_{n},a)}{t_{n}}=\underset{n\rightarrow \infty }{\lim
}\frac{d(x_{n}^{\ast },a)}{t_{n}}=\gamma _{0}>0.
\end{equation*}%
Hence each from the sequences $\tilde{x},\tilde{y},\tilde{z},%
\tilde{w},\tilde{x}^{\ast },$ $\tilde{z}^{\ast }$ is mutually
stable with $\tilde{a}.$  Consequently, by Lemma \ref{5:l2.6},
there are $\tilde d_{\tilde t}(\tilde x, \tilde z)$, $\tilde
d_{\tilde t}(\tilde y, \tilde w)$ and $\tilde d_{\tilde t}(\tilde
x^{*}, \tilde z^{*})$. Moreover \eqref{5:eq2.19},
\eqref{5:eq2.18}, \eqref{5:eq2.22} and \eqref{5:eq2.23} imply that
\begin{equation*}
1+\gamma_0\geq\tilde d_{\tilde t}(\tilde x,\tilde z)\geq\tilde
d_{\tilde t}(\tilde y,\tilde w)\geq1-\gamma_0 > 0.
\end{equation*}
It follows from \eqref{5:eq2.22}, \eqref{5:eq2.23}  and \eqref{5:eq2.17} that%
\begin{equation}
0<s_{0}=\underset{n\rightarrow \infty }{\lim }\frac{\Delta
(S_{a}(r_{n}),S_{a}(t_{n}))}{\delta
(S_{a}(r_{n}),S_{a}(t_{n}))}=\underset{n\rightarrow \infty }{\lim
}\frac{d(x_{n},z_{n})}{d(y_{n},w_{n})}=\frac{
\tilde{d}_{\tilde{t}}(\tilde{x},\tilde{z})}{\tilde{d}_{
\tilde{t}}(\tilde{y},\tilde{w})}.  \label{5:eq2.24}
\end{equation}%
Since $\tilde d_{\tilde t}(\tilde y, \tilde w)\neq0$, there is a
finite limit
\begin{equation*}
\underset{n\rightarrow \infty }{\lim }\frac{d(x_{n}^{\ast
},z_{n}^{\ast})}{d(y_{n},w_{n})}.
\end{equation*}%
In particulary it follows from the definitions of $\tilde{x}^{\ast
}$, $\tilde{z}^{\ast}$ that
\begin{equation*}
\underset{n\rightarrow \infty }{\lim }\frac{d(x_{2n+1}^{\ast
},z_{2n+1}^{\ast })}{d(y_{2n+1},w_{2n+1})}=\underset{n\rightarrow
\infty}{\lim }\frac{d(y_{2n+1},w_{2n+1})}{d(y_{2n+1},w_{2n+1})}=1,
\end{equation*}%
moreover, using \eqref{5:eq2.24} we obtain%
\begin{equation*}
\underset{n\rightarrow \infty }{\lim }\frac{d(x_{2n}^{\ast
},z_{2n}^{\ast })}{d(y_{2n},w_{2n})}=\underset{n\rightarrow \infty
}{\lim }\frac{d(x_{2n},z_{2n})}{d(y_{2n},w_{2n})}=s_{0}.
\end{equation*}%
Consequently the equality $s_{0}=1$ holds also for the case where
$0<\gamma _{0}<1$.

(iii) Let $\left\{ (q_{n},t_{n})\right\}_{n\in \mathbb{N}}$ be a
sequences of elements
of $R_{\varepsilon }^{2}$ such that $\underset{n\rightarrow \infty }{\lim }%
(q_{n},t_{n})\break=0$ and \eqref{5:eq2.92} holds. If in
\eqref{5:eq2.92} $c_{0}=0$ or $c_{0}=\infty ,$ then it is clear
that \eqref{5:eq2.91} holds with $\varkappa _{0}=1,$ so it is
sufficient to take
\begin{equation}
0<c_{0}<\infty .  \label{5:eq22.27}
\end{equation}%
Consider the sequence $\tilde{q}=\left\{ q_{n}\right\}
_{n\in\mathbb{N}}$ as a normalizing sequence. Let
$\tilde{x}=\left\{ x_{n}\right\} _{n\in\mathbb{N}}$ and
$\tilde{y}=\left\{ y_{n}\right\} _{n\in\mathbb{N}}$ belong to
$\tilde{X}$ and $d(a,x_{n})=q_{n}~,~d(a,y_{n})=t_{n}$ and
\begin{equation}
\underset{n\rightarrow \infty }{\lim }\frac{d(x_{n},y_{n})}{\Delta
(S_{a}(q_{n}),S_{a}(t_{n}))}=1.  \label{5:eq22.28}
\end{equation}%
Conditions \eqref{5:eq2.92} and \eqref{5:eq22.27} imply that there
is
\begin{equation*}
\tilde{d}_{\tilde{q}}(\tilde{y},\tilde{a})=\underset{n\rightarrow
\infty }{\lim }\frac{d(y_{n},a)}{q_{n}}=\frac{1}{c_{0}}<\infty .
\end{equation*}%
Hence, by Lemma \ref{5:l2.6}, there is a finite limit
\begin{equation*}
\tilde{d}_{\tilde{q}}(\tilde{x},\tilde{y})=\underset{n\rightarrow
\infty }{\lim }\frac{d(x_{n},y_{n})}{q_{n}}.
\end{equation*}%
Moreover, since $(q_{n},t_{n})\in R_{\varepsilon }^{2}$ for all
$n\in\mathbb{N}$, we have $c_{0}\neq 1$. Consequently, using
\eqref{5:eq22.28} and \eqref{5:eq2.92} we obtain
\begin{multline}
\underset{n\rightarrow \infty }{\lim }\frac{\Delta
(S_{a}(q_{n}),S_{a}(t_{n}))}{\left\vert q_{n}-t_{n}\right\vert
}=\underset{n\rightarrow \infty }{\lim }\frac{d(x_{n},y_{n})\Delta
(S_{a}(q_{n}),S_{a}(t_{n}))}{q_{n}\left\vert 1-\frac{t_{n}}{q_{n}}%
\right\vert d(x_{n},y_{n})}=
\\
=\underset{n\rightarrow \infty }{\lim
}\frac{d(x_{n},y_{n})}{q_{n}}\underset{n\rightarrow \infty }{\lim
}\frac{1}{\left\vert 1-\frac{t_{n}}{q_{n}}\right\vert
}=\frac{c_{0}}{\left\vert 1-c_{0}\right\vert
}\tilde{d}_{\tilde{q}}(\tilde{x},\tilde{y}). \label{5:eq22.29}
\end{multline}%

 Suppose that conditions (i)--(iii) are satisfied simultaneously. We
must to prove that $\Omega_{a,\tilde r}$ is unique for every
normalizing sequence $\tilde r$. Let $\tilde{r}=\left\{
r_{n}\right\} _{n\in\mathbb{N}}$ be an arbitrary normalizing
sequence and let $\tilde{x}=\left\{ x_{n}\right\}
_{n\in\mathbb{N}}$ and $\tilde{y}=\left\{ y_{n}\right\}
_{n\in\mathbb{N}}$ be two elements of $\tilde{X}$ such that
\begin{equation*}
0\leq \tilde{d}(\tilde{a},\tilde{y})=\underset{n\rightarrow \infty
}{\lim }\frac{d(a,x_{n})}{r_{n}}<\infty
\end{equation*}%
and%
\begin{equation*}
0\leq \tilde{d}(\tilde{a},\tilde{y})=\underset{n\rightarrow \infty
}{\lim }\frac{d(a,y_{n})}{r_{n}}<\infty .
\end{equation*}%
To prove the uniqueness of $\Omega_{a,\tilde r}$ it is sufficient,
by Lemma \ref{5:l2.6}, to show that $\tilde{x}$ \ and $\tilde{y}$
are mutually stable w.r.t. $\tilde{r}.$ If
$\tilde{d}(\tilde{a},\tilde{x})=0$, then, by the triangle
inequality,
\begin{equation*}
\underset{n\rightarrow \infty }{\lim \sup }\frac{d(x_{n},y_{n})}{r_{n}}\leq ~%
\underset{n\rightarrow \infty }{\lim}(\frac{d(x_{n},a)}{r_{n}}+\frac{%
d(y_{n},a)}{r_{n}})= \tilde{d}(\tilde{a},\tilde{y})
\end{equation*}%
and%
\begin{equation*}
\underset{n\rightarrow \infty }{\lim \inf }\frac{d(x_{n},y_{n})}{r_{n}}\geq ~%
\underset{n\rightarrow \infty }{\lim}(\frac{d(y_{n},a)}{r_{n}}-\frac{%
d(x_{n},a)}{r_{n}})=\tilde{d}(\tilde{a},\tilde{y}).
\end{equation*}%
Consequently, there is a finite limit
\begin{equation*}
\tilde{d}(\tilde{x},\tilde{y})=\underset{n\rightarrow \infty}{\lim
}\frac{d(x_{n},y_{n})}{r_{n}}=d(\tilde{a},\tilde{y}),
\end{equation*}%
i.e.,  $\tilde{x}$ \ and $\tilde{y}$ are mutually stable. The case
where $d(\tilde{a},\tilde{y})=0$ is similar. Hence, without loss
of generality we may assume that
\begin{equation*}
\tilde{d}(\tilde{a},\tilde{y})\neq 0\neq d(\tilde{a},\tilde{x}).
\end{equation*}%
Consider first the case where
\begin{equation*}
\tilde{d}(\tilde{a},\tilde{y})=d(\tilde{a},\tilde{x}):=b\neq 0.
\end{equation*}%
This assumption implies that for every $k>1$ there is
$n_{0}=n_{0}(k)\in\mathbb{N}$ such that the inclusion
\begin{equation}
A_{a}(b{r_{n}},k)\supseteq \left\{ x_{n},y_{n}\right\}
\label{5:eq2.25}
\end{equation}%
holds for all natural $n>n_{0}(k),$ where%
\begin{equation*}
A_{a}(b{r_{n}},k)=\left\{ x\in X:\frac{b{r_{n}}}{k}\leq d(x,a)\leq
kb{r_{n}}\right\}.
\end{equation*}%
It follows from \eqref{5:eq2.25} that%
\begin{equation*}
d(x_{n},y_{n})\leq \text{diam}(A_{a}(b{r_{n}},k))
\end{equation*}%
if $n>n_{0}(k).$ Consequently
\begin{equation*}
\frac{1}{b}\underset{n\rightarrow \infty }{\lim \sup
}\frac{d(x_{n},y_{n})}{r_{n}}\leq ~\underset{n\rightarrow \infty
}{\lim \sup }\frac{\text{diam(}A_{a}(b{r_{n}},k)\text{)}}{br_{n}}.
\end{equation*}%
Letting $k\rightarrow 1$ on the right-hand side of the last
inequality and using \eqref{5:eq2.8} we see that
\begin{equation*}
0\leq \frac{1}{b}\underset{n\rightarrow \infty }{\lim \sup }\frac{%
d(x_{n},y_{n})}{r_{n}}\leq \underset{k\rightarrow \infty }{\lim
}\left(
\underset{n\rightarrow \infty }{\lim \sup }\frac{\text{diam(}%
A_{a}(b{r_{n}},k)\text{)}}{br_{n}}\right)=0.
\end{equation*}%
Hence%
\begin{equation}\label{5:eq2.30}
\tilde d (\tilde x, \tilde y)=\underset{n\rightarrow \infty }{\lim
}\frac{d(x_{n},y_{n})}{r_{n}}=0.
\end{equation}
It implies that $\tilde{x}$  and $\tilde{y}$ are mutually stable.
It still remains to show that there exists a finite limit
\begin{equation*}
\tilde{d}(\tilde{x},\tilde{y})=\underset{n\rightarrow \infty
}{\lim }\frac{d(x_{n},y_{n})}{r_{n}}
\end{equation*}
if
\begin{equation}
0\neq \tilde{d}(\tilde{x},\tilde{a})\neq \tilde{d}(
\tilde{y},\tilde{a})\neq 0.  \label{5:eq2.32}
\end{equation}%
For convenience we write%
\begin{equation*}
q_{n}:=d(x_{n},a)~,~t_{n}:=d(y_{n},a)
\end{equation*}%
for all $n\in\mathbb{N}$. Condition \eqref{5:eq2.32} implies that
there are $\varepsilon >0$ and a natural number
$n_{0}=n_{0}(\varepsilon)$ such that
\begin{equation}
q_{n}\wedge t_{n}>0\text{ and }\left\vert
\frac{q_{n}}{t_{n}}-1\right\vert \geq \varepsilon \label{5:eq2.33}
\end{equation}%
for all $n\geq n_{0}.$ It is clear that
\begin{equation*}
x_{n}\in S_{a}(q_{n})\text{ and }y_{n}\in S_{a}(t_{n})
\end{equation*}%
where $S_{a}(q_{n})$ and $S_{a}(t_{n})$ are the spheres with the
common center $a\in X$ and radiuses $q_{n}$, $t_{n}$ respectively.
Consequently we have the
following inequalities%
\begin{equation}
\Delta (S_{a}(q_{n}),S_{a}(t_{n}))\geq d(x_{n},y_{n})\geq \delta
(S_{a}(q_{n}),S_{a}(t_{n})).  \label{5:eq2.34}
\end{equation}%
Limit relations \eqref{5:eq2.9} and \eqref{5:eq2.91} imply that
\begin{equation*}
\varkappa _{0}=\underset{n\rightarrow \infty }{\lim }\frac{\Delta
(S_{a}(q_{n}),S_{a}(t_{n}))}{\left\vert q_{n}-t_{n}\right\vert
}=\underset{n\rightarrow \infty }{\lim }\frac{\delta
(S_{a}(q_{n}),S_{a}(t_{n}))}{\left\vert q_{n}-t_{n}\right\vert}.
\end{equation*}
Hence, using \eqref{5:eq2.34}, we obtain
\begin{equation*}
\varkappa _{0}=\underset{n\rightarrow \infty }{\lim
}\frac{d(x_{n},y_{n})}{\left\vert q_{n}-t_{n}\right\vert
}=\frac{1}{\left\vert
\tilde{d}(\tilde{x},\tilde{a})-\tilde{d}(\tilde{y},\tilde{a}
)\right\vert }\underset{n\rightarrow \infty }{\lim
}\frac{d(x_{n},y_{n})}{r_{n}}.
\end{equation*}%
Hence%
\begin{equation}
\tilde{d}(\tilde{x},\tilde{y})=\underset{n\rightarrow \infty
}{\lim }\frac{d(x_{n},y_{n})}{r_{n}}=\varkappa _{0}\left\vert
\tilde{d}(\tilde{x},\tilde{a})-\tilde{d}(\tilde{y},\tilde{a})\right\vert
, \label{5:eq2.35}
\end{equation}%
i.e.,  $\tilde{x}$ and $\tilde{y}$ are mutually stable.
\end{proof}
The initial version of Theorem \ref{5:t2.5} was published in
\cite{DAK1}.

The following proposition will be helpful in the future.
\begin{proposition}\label{p:3.8}
Let $(X,d)$ be a metric space with a marked point $a$, let
$Y\subseteq X$ and let $a\in Y$. If pretangent space
$\Omega_{a,\tilde r}^X$ is unique for every normalizing sequence
$\tilde r$, then pretangent space $\Omega_{a,\tilde r}^Y$ is
unique for every $\tilde r$.
\end{proposition}
\begin{proof}
Apply Lemma \ref{5:l2.6}.
\end{proof}

\section{Future examples of metric spaces with the unique pretangent spaces}

Using Example \ref{4:e2.1} as the simplest model we can construct
some more interesting from the geometric point of view examples of
metric spaces with unique tangent spaces. To this end we recall
first some facts related to the structure of pretangent spaces to
subspaces of metric spaces.

Let $(X,d)$ be a metric space  with a marked point $a$, let $Y$
and $Z$ be subspaces of $X$ such that $a\in Y\cap Z$ and let
$\tilde r=\{r_n\}_{n\in\mathbb N}$ be a normalizing sequence.
\begin{definition}\label{6:d5.1}
The subspaces $Y$ and $Z$ are {\it tangent equivalent} at the
point $a$ w.r.t. the normalizing sequence $\tilde r$ if for every
$\tilde y_1=\{y_n^{(1)}\}_{n\in\mathbb N}\in\tilde Y$ and every
$\tilde z_1=\{z_n^{(1)}\}_{n\in\mathbb N}\in\tilde Z$ with finite
limits
$$
\tilde d_{\tilde r}(\tilde a, \tilde
y_1)=\lim_{n\to\infty}\frac{d(y_n^{(1)},a)}{r_n}\quad\text{and}\quad
\tilde d_{\tilde r}(\tilde a, \tilde
z_1)=\lim_{n\to\infty}\frac{d(z_n^{(1)},a)}{r_n}
$$
there exist $\tilde y_2=\{y_n^{(2)}\}_{n\in\mathbb N}\in\tilde Y$
and $\tilde z_2=\{z_n^{(2)}\}_{n\in\mathbb N}\in\tilde Z$ such
that
$$
\lim_{n\to\infty}\frac{d(y_n^{(1)},z_n^{(2)})}{r_n}=\lim_{n\to\infty}
\frac{d(y_n^{(2)},z_n^{(1)})}{r_n}=0.
$$
\end{definition}
We shall say that $Y$ and $Z$ are {\it strongly tangent
equivalent} at $a$ if $Y$ and $Z$ are tangent equivalent at $a$
for all normalizing sequences $\tilde r$.

Let  $\tilde F\subseteq\tilde X$. For a normalizing sequence
$\tilde r$ we define a family $[\tilde F]_Y=[\tilde F]_{Y,\tilde
r}$ by the rule
\begin{equation*}\label{6:eq5.1}
(\tilde y\in[\tilde F]_Y)\Leftrightarrow((\tilde y\in\tilde
Y)\&(\exists\,\tilde x\in\tilde F:\tilde d_{\tilde r}(\tilde
x,\tilde y)=0)).
\end{equation*}
The following two lemmas were proved in \cite{Dov}, see also
\cite{DM}.
\begin{lemma}\label{6:p5.2}
Let $Y$ and $Z$ be subspaces of a metric space $X$ and let $\tilde
r$ be a normalizing sequence. Suppose that $Y$ and $Z$ are tangent
equivalent (w.r.t. $\tilde r$) at a point $a\in Y\cap Z$. Then
following statements hold for every maximal self-stable (in
$\tilde Z$) family $\tilde Z_{a,\tilde r}$.
\begin{itemize}
\item[$(i)$] The family $[\tilde Z_{a,\tilde r}]_Y$ is maximal
self-stable (in $\tilde Y$) and we have the equalities
\begin{equation*}\label{6:eq5.2}
[[\tilde Z_{a,\tilde r}]_Y]_Z=\tilde Z_{a,\tilde r}=[\tilde
Z_{a,\tilde r}]_Z.
\end{equation*}

\item[$(ii)$] If $\Omega^Z_{a,\tilde r}$ and $\Omega^Y_{a,\tilde
r}$ are metric identifications of $\tilde Z_{a,\tilde r}$ and,
respectively, of $\tilde Y_{a,\tilde r}:=[\tilde Z_{a,\tilde
r}]_Y$, then the mapping
\begin{equation*}\label{6:eq5.3}
\Omega_{a,\tilde
r}^Z\ni\alpha\longmapsto[\alpha]_Y\in\Omega_{a,\tilde r}^Y
\end{equation*}
is an isometry. Furthermore if $\Omega_{a,\tilde r}^Z$ is tangent,
then $\Omega^Y_{a,\tilde r}$ also is tangent.

\item[$(iii)$] Moreover, if for the normalizing sequence $\tilde
r$ here exists a unique maximal self-stable (in $\tilde Z$) family
$\tilde Z_{a,\tilde r}\ni\tilde a$, then $\tilde Y_{a,\tilde
r}:=[\tilde Z_{a,\tilde r}]_Y$ is a unique maximal self-stable (in
$\tilde Y_{a,\tilde r}$) family which contains $\tilde a$.
\end{itemize}
\end{lemma}
Let $Y$ be a subspace of a metric space $(X,d)$. For $a\in Y$ and
$t>0$ we denote by
$$
S_t^Y=S^Y(a,t):=\{y\in Y:d(a,y)=t\}
$$
the sphere (in the subspace $Y$) with the center $a$ and the
radius $t$. Similarly for $a\in Z\subseteq X$ and $t>0$ define
$$
S_t^Z=S^Z(a,t):=\{z\in Z:d(a,z)=t\}.
$$
Write
\begin{equation*}\label{6:eq5.6}
\varepsilon_a(t,Z,Y):=\sup_{z\in S_t^Z}\inf_{y\in Y}d(z,y)
\end{equation*}
and
\begin{equation*}\label{6:eq5.7}
\varepsilon_a(t)=\varepsilon_a(t,Z,Y)\vee\varepsilon_a(t,Y,Z).
\end{equation*}

\begin{lemma}\label{6:t5.4}
Let $Y$ and $Z$ be subspaces of a metric space $(X,d)$ and let
$a\in Y\cap Z$. Then $Y$ and $Z$ are strongly tangent equivalent
at the point $a$ if and only if the equality
\begin{equation}\label{6:eq5.8}
\lim_{t\to0}\frac{\varepsilon_a(t)}t=0
\end{equation}
holds.
\end{lemma}

Using  Proposition \ref{4:p2.2}, Lemma \ref{6:p5.2} and Lemma
\ref{6:t5.4} we can easily obtain  examples of  subspaces of the
Euclidean space which have unique tangent spaces. The first
example will be examined in details.
\begin{example}\label{e3.1}
Let $F:[0,1]\to E^n,\ n\geq 2$, be a Jordan curve in the Euclidean
space $E^n$, i.e., $F$ is continuous  and
$$
F(t_1)\ne F(t_2)
$$
for every two distinct points $t_1,t_2\in[0,1]$. We can write $F$
in the  coordinate form
$$
F(t)=(f_1(t),\dots,f_n(t)),\qquad t\in[0,1].
$$
Suppose that all functions $f_i,\ 1\leq i\leq n$, are
differentiable at the point $0$ and
$$
F'(0)=(f_1'(0),\dots,f_n'(0))\ne(0,\dots,0).
$$
(We use the one-sided derivatives here.) We claim that each
pretangent space to the subspace $Y=F([0,1])\subseteq E^n$ at the
point $a=F(0)$ is unique and tangent and isometric to $\mathbb
R^+$ for every normalizing sequence $\tilde r$. Indeed, by Lemma
\ref{6:p5.2} and by Proposition \ref{4:p2.2}, it is sufficient to
show that $Y$ is strongly tangent equivalent to the ray
$$
Z=\{(z_1(t),\dots,z_n(t)):(z_1(t),\dots,z_n(t))=tF'(0)+F(0),\
t\in\mathbb R^+\}
$$
at the point $a=F(0)$.

The classical definition of the differentiability of real
functions shows that limit relation \eqref{6:eq5.8} holds with
these $Y$ and $Z$. Hence, by Lemma~\ref{6:t5.4}, $Y$ and $Z$ are
strongly tangent equivalent at the point $a=F(0)$.
\end{example}

\begin{example}\label{e3.2}
Let $f_i:[0,1]\to\mathbb R,\ i=1,\dots,n$, be functions such that
$f_1(0)=\dots=f_n(0)=c$ where $c\in\mathbb R$ is a constant.
Suppose all $f_i$ have a common finite right derivative $b$ at the
point $0$, $f_1'(0)=\dots=f_n'(0)=b$. Write
$$
a=(0,c)\quad\text{and}\quad
X=\bigcup_{i=1}^n\{(t,f_i(t)):t\in[0,1]\},
$$
i.e., $X$ is an union of the graphs of the functions $f_i$. Let us
consider $X$ as a subspace of the Euclidean plane $E^2$. Then for
every normalizing sequence $\tilde r$ a pretangent space $\tilde
\Omega_{a,\tilde r}$ to the space $X$ at the point $a$ is unique,
tangent and isometric to $\mathbb R^+$.
\end{example}

\begin{example}\label{e3.3}
Let $f_1,f_2$ be two functions from the precedent example. Put
$$
X=\{(x,y):f_1(x)\wedge f_2(x)\leq y\leq f_1(x)\vee f_2(x),\
x\in[0,1]\},
$$
i.e., $X$ is the set of points of the plane which lie between the
graphs of the functions $f_1$ and $f_2$. Then for every
nozmalazing sequence $\tilde r$ each pretangent space $\tilde
\Omega_{a,\tilde r}$ to $X$ at $a=(0,c)$ is unique, tangent and
isometric to $\mathbb R^+$.
\end{example}
\begin{example}\label{e3.4}
Let $\alpha$ be a positive real number. Write
$$
X=\{(x,y,z)\in E^3:\sqrt{y^2+z^2}\leq x^{1+\alpha},\ x\in\mathbb
R^+\},
$$
i.e., $X$ can be obtained by the rotation of the plane figure
$\{(x,y)\in E^2:0\leq y\leq x^{1+\alpha},\ x\in\mathbb R^+\}$
around the real axis.  Then each pretangent space $\tilde
\Omega_{a,\tilde r}$ to $X$ at the point $a=(0,0,0)$ is unique,
tangent and isometric to $\mathbb R^+$.
\end{example}

In the next our example we will describe the tangent space to the
Cantor set $C$ at the point $0$ w.r.t. the normalizing sequence
$\tilde r=\{\frac1{3^n}\}_{n\in\mathbb N}$. We recall the
definition of the Cantor set $C$. Let $x\in[0,1]$ and expand $x$
as
\begin{equation}\label{eq3.2}
x=\sum_{n=1}^\infty\frac{a_{n_x}}{3^n},\quad a_{n_x}\in\{0,1,2\}.
\end{equation}
The Cantor set $C$ is the set of all points from $[0,1]$ which
have expansion \eqref{eq3.2} using only the digits $0$ and $2$.

Define a set $C^e$ as the smallest subset of $\mathbb R$ which
contains the Cantor set $C$ and satisfies the equality
\begin{equation}\label{eq3.3}
C^e=3^nC^e
\end{equation}
for every integer $n\in\mathbb Z$ where
$$
3^nC^e:=\{3^nx:x\in C^e\}.
$$
It follows from \eqref{eq3.2} that a real number $t$ belongs to
$C^e$ if and only if $t$ has a base $3$ expansion with the digits
0 and 2 only, i.e.,
\begin{equation}\label{eq3.4}
t=\sum_{j=-\infty}^Ma_{n_t}3^j.
\end{equation}
with $M\in\mathbb Z$ and $a_{n_t}\in\{0,2\}$.
\begin{proposition}\label{p:3.9}
Let $X=C$ be the Cantor set with the usual metric $|\cdot,\cdot|$
and let $\tilde r=\{3^{-n}\}_{n\in\mathbb N}$. Then pretangent
space $\Omega_{0,\tilde r}^X$ is unique, tangent and isometric to
$(C^e,|\cdot,\cdot|)$.
\end{proposition}
\begin{proof}
Let $\tilde X_{0,\tilde r}$ be a maximal self-stable family for
which $p(\tilde X_{0,\tilde r})=\Omega^X_{0,\tilde r}$, see
diagram \eqref{1:eq1.4}. The uniqueness of $\tilde X_{0,\tilde r}$
and of $\Omega_{0,\tilde r}$ follows form Proposition \ref{p:3.8}.
As in the proof of Proposition \ref{4:p2.2} we see that for every
$\tilde x=\{x_n\}_{n\in\mathbb N}\in\tilde X_{0,\tilde r}$ there
exists a finite limit
\begin{equation}\label{eq3.5}
\lim_{n\to\infty}\frac{x_n}{3^{-n}}=c(\tilde x)
\end{equation}
and that the function $f:\Omega_{0,\tilde r}\to\mathbb R^+$ with
\begin{equation}\label{eq3.6}
f(\beta)=c(\tilde x)\quad\text{for}\quad \tilde
x\in\beta\in\Omega_{0,\tilde r}
\end{equation}
is well-defined and distance-preserving. Consequently
$\Omega_{0,\tilde r}^X$ is isometric to\linebreak
$(C^e,|\cdot,\cdot|)$ if the following two statements hold:

(i) $c(\tilde x)$ belongs to $C^e$ for every $\tilde x\in\tilde
X_{0,\tilde r}$;

(ii) For every $t\in C^e$ there is $\tilde x\in\tilde X_{0,\tilde
r}$ such that $c(\tilde x)=t$.

To prove Statement (i) note that
\begin{equation}\label{eq3.7}
C^e=\bigcup_{i=0}^\infty 3^iC
\end{equation}
and that for every $t>0$ we have the equality
\begin{equation}\label{eq3.8}
[0,t]\cap3^iC=[0,t]\cap3^jC
\end{equation}
if $i\wedge j\geq\log_3t$. Since $C$ is closed, equality
\eqref{eq3.8} and \eqref{eq3.7} imply that $C^e$ also is closed.
More over, using \eqref{eq3.2}--\eqref{eq3.4} we see that
$$
\frac{x_n}{3^{-n}}\in C^e
$$
for all $x_n\in C$ and all $n\in\mathbb N$. Hence $c(\tilde x)$
belongs to $C^e$ for every $\tilde x\in\tilde X_{0,\tilde r}$,
that is Statement (i) follows.

Let $t$ be an arbitrary point of $C^e$. Then $3^{-n}t\in C$ if
$n>M$, see \eqref{eq3.4}. Write
$$
x_n:=\begin{cases} 0&\text{if }n\leq M\\
3^{-n}t&\text{if }n>M
\end{cases}
$$
for $n\in\mathbb N$
and define $\tilde x:=\{x_n\}_{n\in\mathbb N}$. It is clear that
$c(\tilde x)=t$, so Statement (ii) is true.

It still remains to prove, that $\Omega_{0,\tilde r}^X$ is
tangent. Let $\{n_k\}_{k\in\mathbb N}$ be an infinite strictly
increasing sequence of natural numbers and let $\tilde
r':=\{3^{-n_k}\}_{k\in\mathbb N}$. As in the proof of Statement
(i) we see that the equivalence
$$
\big(\tilde x=\{x_k\}_{k\in\mathbb N}\in\tilde X_{0\tilde
r'}\big)\Leftrightarrow\Big(\lim_{k\to
\infty}\frac{x_k}{3^{-n_k}}\in C^e\Big)
$$
holds for every $\tilde x\in\tilde X$. By Statement (ii) we have
$f(\Omega_{0,\tilde r}^X)=C^e$ where $f$ is defined in
\eqref{eq3.6}. Consequently a function $em':\Omega_{0,\tilde
r}^X\to\Omega_{0,\tilde r'}$, see \eqref{1:eq1.4}, is surjective.
Hence, by Proposition \eqref{1:p1.5}, $\Omega_{0,\tilde r}$ is
tangent.
\end{proof}
Let for $x\in\mathbb R$
\begin{equation}\label{eq3.9}
\varphi_0(x):=\frac13x\quad\text{and}\quad\varphi_1(x)=\frac13x+\frac23.
\end{equation}
It is well known that the Cantor set is the unique nonempty
compact subset of $\mathbb R$ for which
\begin{equation}\label{eq3.10}
 X=\varphi_0( X)\cup\varphi_1(X).
\end{equation}
\begin{theorem}\label{t:3.10}
Let $X$ be the unique nonempty compact subset of $\mathbb R$ for
which equality \eqref{eq3.10} holds, let $m=0,1$, and $a_m$ be
fixed points and $k_m$ be ratios of similarities $\varphi_m$, see
\eqref{eq3.9}, and $\tilde r_m:=\{(k_m)^n\}_{n\in\mathbb N}$. Let
us define the sets $C_m$ and $C_m^e,\ m=0,1$ by the rules
$$
C_m=\{t-a_m:t\in C\},\quad C_m^e=\bigcup_{j\in\mathbb Z}(k_m)^jC_m
$$
where $C$ is the Cantor set. Then for $m=0,1$ the pretangent
spaces $\Omega_{a_m,\tilde r_m}^X$ is unique, tangent and
isometric to $(C_m^e,|\cdot,\cdot|)$.
\end{theorem}
\begin{proof}
The theorem follows from Proposition \ref{p:3.9} because
$C_0,C_1,C$ are isometric and $C_0^e,C_1^e,C^e$ are isometric and
$\{(k_0)^n\}_{n\in\mathbb N}=\{(k_1)^n\}_{n\in\mathbb
N}=\{3^{-n}\}_{n\in\mathbb N}$ and $a_0=0=a_1-1$.
\end{proof}
\begin{remark}\label{r:3.11}
Certainly, Theorem \ref{t:3.10} is, on the whole, a reformulation
of Proposition \ref{p:3.9} but in this form the result admits
generalizations for invariant sets
$$
K=f_0(K)\cup f_1(K)\cup\dots\cup f_n(K)
$$
of some other iterated function systems $(f_0,\dots,f_n)$.
\end{remark}
In all examples above pretangent spaces $\Omega_{a,\tilde r}^X$
were also tangent. The following example shows that there is a
metric space $X$ for which $\Omega_{a,\tilde r}^X$ is unique but
not tangent.
\begin{example}\label{e3.12}
Let $\tilde r=\{r_n\}_{n\in\mathbb N}$ be a sequence of strictly
decreasing positive real numbers $r_n$ with
\begin{equation}\label{eq3.10*}
\lim_{n\to\infty}\frac{r_n}{r_{n+1}}=\infty
\end{equation}
and such that $r_n>2r_{n+1}$ for all $n\in\mathbb N$. Let $X$ be a
union of two countable sets $\{r_n:n\in\mathbb N\}$ and
$\{2r_{2n}:n\in\mathbb N\}$
 and the one-point set $\{0\}$ ,
\begin{equation}\label{eq3.12}
X=\{r_n:n\in\mathbb N\}\cup\{2r_{2n}:n\in\mathbb N\}\cup\{0\}.
\end{equation}
Consider the metric space $(X,|\cdot,\cdot|)$. It is clear that
the sequences $\tilde 0$  and $\tilde x:=\{r_n\}_{n\in\mathbb N}$
are mutually stable w.r.t. $\tilde r$ and
$$
\tilde d_{\tilde r}(\tilde x,\tilde 0)=1.
$$
Let $\tilde X_{0,\tilde r}$ be a unique (by Proposition
\ref{p:3.8}) maximal self-stable family such that
$$
\tilde X_{0,\tilde r}\supseteq \{0,\tilde x\}.
$$
We claim that the pretangent space $\Omega_{0,\tilde r}^X$
corresponding to $\tilde X_{0,\tilde r}$ is two-point. Indeed,
suppose that $\tilde y=\{y_n\}_{n\in\mathbb N}\in\tilde
X_{0,\tilde r}$ and $\tilde d(\tilde y,\tilde 0)>0$. It is
sufficient to prove that the equality
\begin{equation}\label{eq3.13}
\tilde d(\tilde x,\tilde y)=0.
\end{equation}
holds. To this end, we note that \eqref{eq3.10*} and
\eqref{eq3.12} imply
\begin{equation}\label{eq3.14}
\frac{y_{2n+1}}{r_{2n+1}}=1\quad
\text{and}\quad\frac{y_{2n}}{r_{2n}}\in\{1,2\}
\end{equation}
for all sufficiently large $n\in\mathbb N$ because in the opposite
case
$$
\text{either }\lim_{n\to\infty}\frac{y_n}{r_n}=0\quad\text{or
}\lim_{n\to\infty}\frac{y_n}{r_n}=\infty.
$$
Since
$$
1=\lim_{n\to\infty}\frac{y_{2n+1}}{r_{2n+1}}=\lim_{n\to\infty}\frac{y_n}{r_n}=\lim_{n\to\infty}
\frac{y_{2n}}{r_{2n}},
$$
conditions \eqref{eq3.14} imply that
$$
y_{2n}=r_{2n}
$$
for sufficiently large $n$, Hence \eqref{eq3.13} follows.

Now let $\tilde r':=\{r_{2n}\}_{n\in\mathbb N}$ and $\tilde
X_{0,\tilde r'}$ be a maximal self-stable family such that
$$
\tilde X_{0,\tilde r'}\supseteq\{\tilde0,\tilde x,\tilde z\}
$$
where $\tilde x:=\{r_{2n}\}_{n\in\mathbb N}$ and $\tilde
z:=\{2r_{2n}\}_{n\in\mathbb N}$. Since
$$
1=\tilde d_{\tilde r'}(\tilde 0,\tilde x)=\frac12\tilde d_{\tilde
r'}(\tilde 0,\tilde z)=\tilde d(\tilde x,\tilde y),
$$
the pretangent space $\Omega_{0,\tilde r'}$ corresponding to
$\tilde X_{0,\tilde r'}$ contains at least three distinct points.
Consequently $\Omega_{0,\tilde r}$ is not tangent.
\end{example}

\bigskip{\bf Acknowledgement.} The initial version
of this paper was produced during the visit of the first author to
the Mersin University (TURKEY) in February-April 2008 under the
support of the T\"{U}B\.{I}TAK-Fellowships For Visiting Scientists
Programme.

\end{document}